\documentclass[letterpaper, 10pt,  reqno]{amsart}

\usepackage{amsmath,amssymb,amscd,amsthm,amsxtra, esint}

\usepackage[margin=1.2in,marginparwidth=1.5cm, marginparsep=0.5cm]{geometry}

\usepackage[mathscr]{euscript}

\usepackage[implicit=true]{hyperref}

\usepackage{dsfont}%for indicator funct.

\newcommand{\R}{\mathbb R}
\newcommand{\Z}{\mathbb Z}
\newcommand{\T}{\mathbb T}
\newcommand{\dd}{\,\mathrm d}
\newcommand{\one}{\mathbf 1}
\newcommand{\norm}[1]{\left\lVert#1\right\rVert}
\newcommand{\Aone}{\mathcal A_1}
\newcommand{\Atwo}{\mathcal A_2}
\DeclareMathOperator*{\esssup}{ess\,sup}
\DeclareMathOperator{\supp}{supp}
\DeclareMathOperator{\dist}{dist}

\hypersetup{hidelinks}

\allowdisplaybreaks[2]

\usepackage{cancel}

\usepackage{color}

\definecolor{gr}{rgb}   {0.,   0.69,   0.23 }
\definecolor{bl}{rgb}   {0.,   0.5,   1. }
\definecolor{mg}{rgb}   {0.85,  0.,    0.85}
\definecolor{yl}{rgb}   {0.8,  0.7,   0.}
\definecolor{or}{rgb}  {0.7,0.2,0.2}

\newtheorem{theorem}{Theorem} [section]
\newtheorem{THM:main}{Theorem}
\newtheorem{lemma}[theorem]{Lemma}
\newtheorem{proposition}[theorem]{Proposition}

\newtheorem{definition}[theorem]{Definition}

\def\norm#1{\|#1\|}

\newcommand{\les}{\lesssim}
\newcommand{\ges}{\gtrsim}

\numberwithin{equation}{section}
\numberwithin{theorem}{section}

\makeatletter
\@namedef{subjclassname@2020}{%
  \textup{2020} Mathematics Subject Classification}
\makeatother

\begin{document}
\baselineskip = 14pt

\title[$L^4$ Strichartz estimate for the Zakharov--Kuznetsov equation]{Sharp $L^4$ Strichartz estimate for the Zakharov--Kuznetsov equation on $\mathbb{R}\times\mathbb{T}$}

\author[E.~Ba\c{s}ako\u{g}lu, W. Hu, Y.~Wang, O. Yılmaz]
{Engin Ba\c{s}ako\u{g}lu, Wenting Hu, Yuzhao Wang and Oğuz Yılmaz}

\address{Engin Ba\c{s}ako\u{g}lu\\Department of Mathematics, Bilkent University, Ankara\\
06800, Türkiye}
\email{ebasakoglu@bilkent.edu.tr} 

\address{Wenting Hu$^{1,2}$\\
$^{1}$ School of Mathematics and Statistics,
Henan Normal University, Xinxiang 453007, P. R. China\\
$^{2}$ School of Mathematics,
University of Birmingham, Watson Building, Edgbaston,
Birmingham B15 2TT, United Kingdom}
\email{huwenting@htu.edu.cn}

\address{Yuzhao Wang\\School of Mathematics\\ University of Birmingham\\ Watson Building\\ Edgbaston\\ Birmingham \\B15 2TT\\ United Kingdom}
\email{y.wang.14@bham.ac.uk}

\address{O\u{g}uz Yılmaz \\ Department of Mathematics, Bo\u gazi\c ci University, 
Bebek 34342, Istanbul, Türkiye}
\email{oguz.yilmaz@bogazici.edu.tr}

\subjclass[2020]{35R01, 37L50}

\keywords{Zakharov--Kuznetsov equation, Strichartz estimate, waveguide manifold,
semi-algebraic set}

\begin{abstract}
We prove the local $L^4$ Strichartz estimate without derivative loss for the Zakharov--Kuznetsov flow on $\R\times\T$. This reaches the sharp spatial regularity endpoint and implies the embedding $X^{0,b}\hookrightarrow L^4$ for every $b>1/2$. The proof reduces the estimate to a semi-algebraic mixed-measure estimate.
\end{abstract}

%\date{\today}

%\vspace*{-5mm}

%%
%
\maketitle

\vspace{-3mm}

\section{Introduction}
\label{sec:intro}

Let $\T=\R/(2\pi\Z)$. We consider the Zakharov--Kuznetsov flow, which is derived from the linear equation
\begin{equation*}
\partial_tu+\partial_x(\partial_x^2+\partial_y^2)u=0,
\qquad (x,y)\in\R\times\T.
\end{equation*}
Its solution operator is
\begin{equation}
\label{eq:flow}
U(t)f=e^{-t\partial_x\Delta}f,
\qquad
\widehat{U(t)f}(\xi,k)=e^{it\phi(\xi,k)}\widehat f(\xi,k),
\qquad
\phi(\xi,k)=\xi(\xi^2+k^2).
\end{equation}
In this paper, we prove the sharp local $L^4$ Strichartz estimate
for the Zakharov--Kuznetsov flow on $\R\times\T$.
Farah and Molinet \cite[Proposition 3.1]{FM} established
\begin{equation}
\label{eq:FM}
\norm{u}_{L^4(\R_t\times\R\times\T)}
\lesssim \norm{u}_{X^{1/6,\,3/8}}.
\end{equation}
More recently, Nowicki-Koth \cite[Theorem 1.1]{NK} improved
the spatial regularity requirement to
\begin{equation}
\label{eq:NK}
\norm{u}_{L^4(\R_t\times\R\times\T)}
\lesssim_{\varepsilon,b}\norm{u}_{X^{\varepsilon,b}},
\qquad \varepsilon>0,\quad b>\tfrac12.
\end{equation}
Here $X^{s,b}$ denotes the Bourgain space associated with
the linear flow \eqref{eq:flow}, equipped with the norm
\begin{equation*}
\norm{u}_{X^{s,b}}
=
\left(
\sum_{k\in\Z}\int_{\R^2}
\langle(\xi,k)\rangle^{2s}
\langle\tau-\phi(\xi,k)\rangle^{2b}
|\widehat u(\tau,\xi,k)|^2
\dd\tau\dd\xi
\right)^{1/2}.
\end{equation*}
Although the modulation exponents $b$ in \eqref{eq:FM} and
\eqref{eq:NK} differ, the latter estimate reduces the
spatial derivative loss $s$ to an arbitrarily small positive
amount $\varepsilon > 0$. Moreover, Nowicki-Koth
\cite[Proposition 3.10]{NK} showed that such an estimate
fails for negative spatial regularity $s < 0$. Thus, the endpoint
$s=0$ remained open. Our main result resolves this
endpoint problem by removing the derivative loss entirely.

\begin{theorem}
\label{thm:main}
For every $f\in L^2(\R\times\T)$,
\begin{equation}
\label{eq:main}
\norm{U(t)f}_{L^4([0,1]\times\R\times\T)}
\lesssim\norm f_{L^2(\R\times\T)}.
\end{equation}
Consequently, for every $b>1/2$,
\begin{equation}
\label{eq:embedding}
\norm u_{L^4(\R_t\times\R\times\T)}
\lesssim_b\norm u_{X^{0,b}}.
\end{equation}
\end{theorem}

The proof reduces the Strichartz estimate to uniform
mixed-measure bounds on $\R\times\Z$, following an approach
used in \cite{BSTW26,BSTW25,DFZ,DWWZ,TT01}.
A central tool is the estimate of Deng, H.~Wang, Y.~Wang
and Zhao \cite[Lemma 4.2]{DWWZ}, which controls the mixed
measure of a semi-algebraic set of bounded complexity
by its Euclidean area and the maximal length of its
sections parallel to the continuous frequency axis.
The same estimate is used in \cite{DFZ} to prove
the sharp $L^4$ Strichartz estimate for the hyperbolic
Schr\"odinger equation on $\R\times\T$.

The additional difficulty in the present problem comes
from the cubic Zakharov--Kuznetsov phase.
The associated four-wave resonance is represented by
the symmetric trilinear form $T$ in \eqref{eq:resonance}.
Fixing one interaction variable $c$ leads to the quadratic
form $Q_c(v)=T(c,v,v)$, whose coefficients depend on $c$.
Uniform estimates for this family must account for two
types of degeneracy. In the coordinates \eqref{eq:RZ},
the form $Q_c$ loses rank when $R_c=0$ or $S_c=0$,
while its quadratic coefficient in the continuous
frequency variable $\xi_v$ vanishes when $R_c+S_c=0$.
These three lines determine the initial decomposition
of the interaction set.

For the remaining interactions, where at most one
of $a,b,c$ is singular, we use a decomposition inspired
by \cite{DFZ} and apply a different pairing of the Fourier
factors on each region.
In the first pairing, $c$ is fixed, and the Euclidean
area estimate in the hyperbolic case follows from
\cite[Proposition 2.4, estimate (2.7)]{DFZ} after
a linear change of variables and scaling.
In the second pairing, we fix $b$ and $s=a+c$,
so the parameter $c=s-a$ varies with the integration
variable $a$.
The normalization used for the first pairing therefore
does not give a direct reduction to \cite{DFZ},
and we establish the required measure bound
by a separate argument.

The rest of the paper is organised as follows.
Section~\ref{sec:prelim} contains the Fourier
reduction and the auxiliary measure estimates.
We prove Theorem~\ref{thm:main} in
Section~\ref{sec:strips}.

\subsection{Notation}

We write $z=(\xi,k)$ for an original frequency, where $\xi\in\R$ and $k\in\Z$, and $\langle z\rangle=(1+|z|^2)^{1/2}$. Integration over frequency space means Lebesgue integration in the continuous coordinate and counting in the discrete coordinate. The notation $|E|_{\R\times\Z}$ denotes this mixed measure, while $|E|_{\R^2}$ denotes Euclidean area.

For a frequency vector $z = (\xi,k) \in \R\times\Z$, its transformed coordinates are
\begin{equation}\label{eq:RZ}
R_z = \sqrt3\,\xi+k,\qquad S_z=\sqrt3\,\xi-k.
\end{equation} 
The relations $A\lesssim B$ and $A\sim B$ denote an inequality up to a constant and two-sided comparability, respectively. Once the strip width $K$ and phase thickness $C_0$ are fixed, constants may depend on them but not on the free frequencies. A dyadic number is an element of $2^\Z$.

\section{Preliminaries}
\label{sec:prelim}

\subsection{Fourier reduction}
\label{SUB:FR}
We first work with smooth data of bounded Fourier support; the estimates extend by approximation. Choose a nonnegative even $\chi\in\mathcal S(\R)$ such that $\chi\ge1$ on $[0,1]$ and $\supp\widehat\chi\subset[-M_0,M_0]$. Such a function is obtained as a sufficiently large multiple of $|\check\rho|^2$, where $\rho\in C_c^\infty(\R)$ is real, nonnegative, and supported sufficiently close to zero.

Expanding $\int\chi(t)|U(t)f|^4\dd t\dd x\dd y$ and taking absolute values gives
\begin{equation}
\label{eq:fourier-reduction}
\norm{U(t)f}_{L^4([0,1]\times\R\times\T)}^4
\lesssim
\int\delta_0(z_1-z_2+z_3-z_4)
\one_{\{|\Omega|\le M_0\}}\prod_{j=1}^4F(z_j)
\dd z_1\cdots\dd z_4,
\end{equation}
where $F=|\widehat f|$, $\delta_0$ denotes the product
of the Dirac delta in the continuous frequency variable
and the Kronecker delta in the discrete frequency
variable, and
\begin{align}\label{eq:omega}
\Omega=\phi(z_1)-\phi(z_2)+\phi(z_3)-\phi(z_4),
\end{align}
where $\phi$ is given in \eqref{eq:flow}.
Parametrize the constraint $z_1-z_2+z_3-z_4=0$ by
\begin{equation}
\label{eq:interaction}
\begin{split}
z_1=a+b+c,\qquad z_2=a+b-c, \qquad
z_3=a-b-c,\qquad z_4=a-b+c.
\end{split}
\end{equation}
Since $\phi(z)=(R_z^3+S_z^3)/(6\sqrt3)$ in view of \eqref{eq:RZ}, the scalar cubic identity
\begin{equation*}
(x+y+z)^3-(x+y-z)^3+(x-y-z)^3-(x-y+z)^3=24xyz\end{equation*}
implies that $\Omega$ in \eqref{eq:omega} can be written as
\begin{equation}
\label{eq:resonance}
\Omega=\frac4{\sqrt3}T(a,b,c),
\qquad
T(a,b,c)=R_aR_bR_c+S_aS_bS_c.
\end{equation}
The inverse map is
$a=(z_1+z_3)/2$, $b=(z_2-z_3)/2$, $c=(z_1-z_2)/2$.
Thus the discrete components of $a,b,c$ lie in $\frac12\Z$.
We extend $F$ by zero from $\R\times\Z$ to
$\R\times\frac12\Z$ and use counting measure on the latter lattice.
By nonnegativity, the discrete variables $a,b,c$ may be summed over
the full product $(\frac12\Z)^3$.

Therefore, to prove Theorem \ref{thm:main}, it suffices to prove the following estimate, with $C_0>0$ fixed.

\begin{proposition}
\label{prop:quartic}
For every nonnegative $F\in L^2(\R\times\Z)$, extended by zero as above,
\begin{equation}
\label{eq:quartic}
\mathcal I(F):=
\iiint\one_{\{|T(a,b,c)|\le C_0\}}
\prod_{j=1}^4F(z_j)\dd a\dd b\dd c
\lesssim\norm F_2^4,
\end{equation}
where $z_j$ are given by \eqref{eq:interaction}.
\end{proposition}

Replacing $F$ by $F(z)+F(-z)$ increases its $L^2$ norm by at most a factor two. We may therefore assume $F$ is even. Permuting $a,b,c$ then leaves the integrand invariant, up to permutations and sign changes of the $z_j$. This symmetry will be used throughout.

We shall use the quadratic form
\begin{equation*}
Q_c(v)=T(c,v,v)=R_cR_v^2+S_cS_v^2.
\end{equation*}
In particular,
\begin{equation}
\label{eq:Q-identities}
Q_c(a+b)-Q_c(a-b)=4T(a,b,c),
\qquad
T(c,a+b,a-b)=Q_c(a)-Q_c(b).
\end{equation}

\subsection{Measure estimates}
We state the main tool of this paper.
\begin{lemma}[{\cite[Lemma 4.2]{DWWZ}}]
\label{lem:mixed}
If $E\subset\R^2$ is a bounded semi-algebraic set of bounded complexity, then
\begin{equation*}
|E|_{\R\times\Z}
\lesssim
|E|_{\R^2}
+\sup_{k\in\Z}|\{\xi:(\xi,k)\in E\}|.
\end{equation*}
The constant depends only on the number and degrees of the defining polynomial inequalities.
\end{lemma}

The same estimate applies to unbounded semi-algebraic
sets by applying it to $E\cap[-N,N]^2$ and letting
$N\to\infty$; these truncations have uniformly bounded
complexity.
After the fixed dilation $k\mapsto2k$, we may also
use the estimate on $\R\times\frac12\Z$.

\medskip 

We then collect some elementary sublevel-set estimates
for quadratic forms.

\begin{lemma}
\label{lem:quadratic}
If $\alpha\ne0$, then, uniformly in $\beta,\gamma\in\R$,
\begin{equation}
\label{eq:quadratic}
|\{x:|\alpha x^2+\beta x+\gamma|\le1\}|
\lesssim|\alpha|^{-1/2}.
\end{equation}
For intervals $I,J$ of lengths $A,B$,
\begin{equation}
\label{eq:rectangle}
|\{(X,Y)\in I\times J:|X^2-Y^2-\gamma|\le1\}|
\lesssim1+\log(2+\min(A,B)).
\end{equation}
If $J$ is separated from zero, then also
\begin{equation}
\label{eq:away}
|\{(X,Y)\in\R\times J:|X^2-Y^2-\gamma|\le1\}|
\lesssim1+\log\left(1+\frac{|J|}{\dist(J,0)}\right).
\end{equation}
All bounds are uniform in the interval positions and in $\gamma$.
\end{lemma}

\begin{proof}
The one-dimensional estimate \eqref{eq:quadratic} is a
standard consequence of the van der Corput sublevel-set
inequality; see, for example, \cite[Lemma 1.2]{CCW}.
The remaining estimates \eqref{eq:rectangle} and \eqref{eq:away} follow from elementary integration.
A similar argument was employed in the proof of
\cite[Theorem 3.1]{BSTW25}.
We omit the details here.
\end{proof}

% Every estimate in Lemma \ref{lem:quadratic} also holds for a different fixed sublevel thickness, with an additional dependence on that thickness.

\begin{lemma} 
\label{lem:strip-shell}
Fix $\kappa,d,\ell>0$. Suppose $|p|,|q|\ge\kappa$, and let $I$ be an interval with $|I|\le\ell$ and $\dist(I,0)\ge d$. Uniformly in $\gamma\in\R$,
\begin{equation}\label{eq:strip}
\left|\left\{(\xi,k):\xi\in I,\
\left|p(\sqrt3\xi+k)^2+q(\sqrt3\xi-k)^2-\gamma\right|\le1
\right\}\right|
\lesssim_{\kappa,d,\ell}1.
\end{equation}
The bound holds for both Euclidean area and mixed $\R\times\Z$ measure.
\end{lemma}

\begin{proof}
If $p+q=0$, the quadratic expression is $4\sqrt3p\xi k$. For each $\xi\in I$, its $k$-section has length $O((|p|\,|\xi|)^{-1})$. Hence its Euclidean area is at most
\begin{equation*}
\frac1{2\sqrt3|p|}\int_I\frac{\dd\xi}{|\xi|}
\le\frac{\ell}{2\sqrt3\kappa d},
\end{equation*}
which proves the Euclidean area bound \eqref{eq:strip} in this case.

If $p+q\ne0$, complete the square:
\begin{equation*}
\begin{split}
p(\sqrt3\xi+k)^2+q(\sqrt3\xi-k)^2
={}&(p+q)\left(k+\frac{\sqrt3(p-q)}{p+q}\xi\right)^2 +\frac{12pq}{p+q}\xi^2.
\end{split}
\end{equation*}
Make the change of variables
\begin{equation*}
X=|p+q|^{1/2}\left(k+\frac{\sqrt3(p-q)}{p+q}\xi\right),
\qquad
Y=\left|\frac{12pq}{p+q}\right|^{1/2}\xi.
\end{equation*}
The inverse Jacobian is $(12|pq|)^{-1/2}$. The transformed form is $\pm X^2\pm Y^2$. In the elliptic case its unit-thickness shell has area at most $2\pi$. In the hyperbolic case the $Y$-interval has length-to-distance ratio at most $\ell/d$, and \eqref{eq:away} applies. 
Then \eqref{eq:strip} with $\R^2$ measure follows by restoring the Jacobian.

At fixed original coordinate $k$, the continuous section lies in $I$ and has length at most $\ell$. Lemma \ref{lem:mixed} proves the mixed bound.
\end{proof}

\section{Proof of the main theorem}
\label{sec:strips}

This section is devoted to the proof of Theorem \ref{thm:main}.
Fix a sufficiently large constant $K\ge10$, and define
\begin{equation*}
\Gamma_R=\{|R|\le K\},\qquad
\Gamma_S=\{|S|\le K\},\qquad
\Gamma_0=\{|\xi|\le K\}.
\end{equation*}
\begin{definition}\label{DEF:reg}
A vector $v$ is called \emph{regular} if
\begin{equation*}
|R_v|>K,\qquad |S_v|>K,\qquad
|R_v+S_v|>2\sqrt{3}K.
\end{equation*}
Otherwise, $v$ is called \emph{singular}.
\end{definition}

We begin by bounding the contribution to $\mathcal I(F)$
from interactions with at least two singular vectors.
By symmetry, we may take these vectors to be $a$ and $b$.

\medskip

We first record a consequence of an existing estimate on a bounded continuous-frequency strip.

\begin{lemma}
\label{lem:bounded-strip}
For fixed $K_1$, if $F$ is supported in $\{|\xi|\le K_1\}$, then
$\mathcal I(F)\lesssim_{K_1,C_0}\norm F_2^4$.
\end{lemma}

\begin{proof}
By \cite[Proposition 3.3]{NK}, we have
\begin{align}\label{eq:L4cut}
\|U(t)f\|_{L^4([-1,1]\times\R\times\T)}
\lesssim \|\langle D_x\rangle^{1/4}f\|_2 \lesssim C_{K_1}\|f\|_2,
\end{align}
provided that $\widehat f$ is supported in $|\xi|\le K_1$.

Choose a nonnegative Schwartz function $\eta$ such that
$\widehat\eta\ge0$ and $\widehat\eta(\omega)\ge1$
for $|\omega|\le4C_0/\sqrt3$.
Taking $\widehat f=F\ge0$, by the Fourier reduction in Subsection \ref{SUB:FR} and \eqref{eq:L4cut} we get
\[
\mathcal I(F)\lesssim
\int_{\R}\eta(t)\|U(t)f\|_{L^4_{x,y}}^4\,dt
\lesssim_{K_1,C_0}\|F\|_2^4,
\]
which completes the proof.
\end{proof}

\subsection{At least two singular vectors}

In this subsection, we consider the contribution to $\mathcal I(F)$ from interactions with at least two singular vectors.

\begin{proposition}
\label{prop:strips}
We have
\begin{align*}
\iiint_{\mathcal S} \one_{\{|T(a,b,c)|\le C_0\}}
\prod_{j=1}^4F(z_j)\dd a\dd b\dd c \les_K \norm F_2^4,
\end{align*}
where $\mathcal S$ is the set where at least two of the vectors \(a,b,c\) are singular.
\end{proposition}

\begin{proof}
We consider three cases.

\noindent\textbf{Case 1.}
Let $\Gamma_i$ and $\Gamma_j$ be two distinct strips
among $\Gamma_R$, $\Gamma_S$ and $\Gamma_0$.
Here we consider the case that $a\in\Gamma_i$ and $b\in\Gamma_j$,
where $i\neq j$. Set $d=a+b$ and $e=a-b$.
By Young's inequality, we have
\begin{equation}
\label{eq:case1-pairing}
\begin{aligned}
\prod_{j=1}^4 F(z_j)
\leq {}&
\frac12 F(d+c)^2F(d-c)^2 +\frac12 F(e+c)^2F(e-c)^2.
\end{aligned}
\end{equation}

We first consider the contribution of the first term
in \eqref{eq:case1-pairing}. Since the integrand is
nonnegative, we may discard the phase restriction
$|T(a,b,c)|\leq C_0$.
Changing variables from $(a,b,c)$ to $(a,d,c)$,
with $b=d-a$, it suffices to estimate
\begin{equation*}
\iint F(d+c)^2F(d-c)^2
\left(
\int \mathbf{1}_{\Gamma_i}(a)
\mathbf{1}_{\Gamma_j}(d-a)\dd a
\right)\dd d\dd c.
\end{equation*}
Thus, it is enough to show that, uniformly in $d$,
\begin{equation}
\label{eq:case1-fibre}
\left|
\left\{
a:a\in\Gamma_i,\ d-a\in\Gamma_j
\right\}
\right|_{\R\times\frac12\Z}
\lesssim_K 1.
\end{equation}

To prove \eqref{eq:case1-fibre}, recall that the three
strip-defining linear functions are
\[
R_a=\sqrt3\xi_a+k_a,\qquad
S_a=\sqrt3\xi_a-k_a,\qquad
\xi_a.
\]
Their gradients in the $(\xi_a,k_a)$-plane are
$(\sqrt3,\pm 1)$ and
$(1,0)$, 
and any two of these vectors are linearly independent.
Since there are only three possible pairs, their determinants
are bounded away from zero. Thus the intersection of any two
translates of distinct strips has uniformly bounded Euclidean area.

For instance, suppose that $a\in\Gamma_R$ and
$d-a\in\Gamma_S$. Then it follows that $|R_a|\leq K$ and $|S_d-S_a|\leq K$.
In the coordinates $(R_a,S_a)$, this region is
contained in a rectangle of side lengths $2K$
and $2K$. Since
\begin{align}\label{eq:det}
\left|
\det\frac{\partial(R_a,S_a)}
{\partial(\xi_a,k_a)}
\right|=2\sqrt3,
\end{align}
its Euclidean area $ |
\{a:a\in\Gamma_R,\ d-a\in\Gamma_S\}
 |_{\R^2} $ is bounded by $C_K$,
uniformly in $d$. The other pairs of distinct
strips are treated in the same way. Therefore,
\begin{equation}
\label{eq:case1-area}
\left|
\left\{
a:a\in\Gamma_i,\ d-a\in\Gamma_j
\right\}
\right|_{\R^2}
\lesssim_K 1.
\end{equation}

Next, fix the discrete coordinate $k_a$.
The condition $a\in\Gamma_i$ restricts $\xi_a$ to an
interval of length $O_K(1)$. Consequently,
\begin{equation}
\label{eq:case1-section}
\sup_{k_a}
\left|
\left\{
\xi_a:
(\xi_a,k_a)\in
\Gamma_i\cap(d-\Gamma_j)
\right\}
\right|
\lesssim_K 1.
\end{equation}

The set $\Gamma_i\cap(d-\Gamma_j)$ is
semi-algebraic of uniformly bounded complexity.
Applying Lemma~\ref{lem:mixed} to
\eqref{eq:case1-area} and \eqref{eq:case1-section}
gives \eqref{eq:case1-fibre}.

The contribution of the second term in
\eqref{eq:case1-pairing} is treated similarly,
using $e=a-b$ in place of $d=a+b$.
This completes the proof of Case 1.

\noindent\textbf{Case 2.}
We now consider the case where $a,b\in\Gamma_R$ or $a,b\in\Gamma_S$. We treat only the former, since the latter follows by interchanging $R$ and $S$.

First consider the region $|S_b|\leq1$.
Together with $|R_b|\leq K$, the identities \eqref{eq:RZ}
show that 
$|\xi_b|\leq\frac{K+1}{2\sqrt3}$ and 
$|k_b|\leq\frac{K+1}{2}$.
Consequently,
\begin{align}\label{eq:case21}
\left|
\left\{
b:|R_b|\leq K,\ |S_b|\leq1
\right\}
\right|_{\R\times\frac12\Z}
\lesssim_K1.
\end{align}
Changing variables from $(a,c)$
to $(v,w)=(a+c,a-c)$, and using the evenness of $F$,
we obtain
\[
\begin{aligned}
\iint \prod_{j=1}^4 F(z_j)\dd a\dd c
&\lesssim
\iint F(v+b)F(v-b)F(w+b)F(w-b)\dd v\dd w\\
&=
\left(\int F(v+b)F(v-b)\dd v\right)^2 \leq \|F\|_2^4,
\end{aligned}
\]
Combining this bound with \eqref{eq:case21} completes the estimate in this region.

Now suppose that $|S_b|>1$.
By \eqref{eq:interaction} and Young's inequality, we have
\begin{equation}
\label{eq:case2-pairing}
\begin{aligned}
\prod_{j=1}^4 F(z_j)
\leq {}&
\frac12 F(a+c+b)^2F(a+c-b)^2 +\frac12 F(a-c+b)^2F(a-c-b)^2.
\end{aligned}
\end{equation}
We first consider the contribution of the first term
in \eqref{eq:case2-pairing}.
Changing variables from $(a,b,c)$ to $(a,b,s)$,
with $s=a+c$, it suffices to estimate
\begin{equation}
\label{eq:case2-integral}
\iint_{\substack{b\in\Gamma_R\\ |S_b|>1}}
F(s+b)^2F(s-b)^2
\left(
\int
\mathbf{1}_{\{|R_a|\leq K\}}
\mathbf{1}_{\{|T(a,b,s-a)|\leq C_0\}}\dd a
\right)\dd b\dd s.
\end{equation}
Thus, it is enough to show that, uniformly in $b,s$
with $|S_b|>1$,
\begin{equation}
\label{eq:case2-fibre}
\left|
\left\{
a:|R_a|\leq K,\ |T(a,b,s-a)|\leq C_0
\right\}
\right|_{\R\times\frac12\Z}
\lesssim_{K,C_0}1.
\end{equation}

To prove \eqref{eq:case2-fibre}, write
\[
E_{b,s}
=
\left\{
a=(\xi_a,k_a)\in\R^2:
|R_a|\leq K,\ |T(a,b,s-a)|\leq C_0
\right\}.
\]
We first estimate its Euclidean area.
In view of \eqref{eq:resonance} and Lemma~\ref{lem:quadratic}, we have 
\[
\left|
\left\{
S_a:
\left|
R_bR_a(R_s-R_a)+S_bS_a(S_s-S_a)
\right|\leq C_0
\right\}
\right|
\lesssim C_0^{1/2}|S_b|^{-1/2}.
\]
By \eqref{eq:det},
integrating over $|R_a|\leq K$ yields
\begin{equation}
\label{eq:case2-area}
|E_{b,s}|_{\R^2}
\lesssim KC_0^{1/2}|S_b|^{-1/2}
\lesssim_{K,C_0}1.
\end{equation}

Next, fix the original discrete coordinate $k_a$.
The condition $|R_a|=|\sqrt3\xi_a+k_a|\leq K$
restricts $\xi_a$ to an interval of length
$2K/\sqrt3$. Consequently,
\begin{equation}
\label{eq:case2-section}
\sup_{k_a\in\frac12\Z}
\left|
\left\{
\xi_a:(\xi_a,k_a)\in E_{b,s}
\right\}
\right|
\leq\frac{2K}{\sqrt3}.
\end{equation}

The set $E_{b,s}$ is semi-algebraic of uniformly
bounded complexity. It is also bounded, since
$R_a$ is bounded and the coefficient of $S_a^2$
is nonzero.
Applying Lemma~\ref{lem:mixed} to \eqref{eq:case2-area}
and \eqref{eq:case2-section} gives
\eqref{eq:case2-fibre}.

Combining \eqref{eq:case2-integral} with \eqref{eq:case2-fibre},
we obtain the desired bound for the contribution of the first term
in \eqref{eq:case2-pairing}.
The second term is treated in the same way.
This completes the proof of Case~2.

\noindent\textbf{Case 3.}
It remains to consider 
$a,b\in\Gamma_0\setminus(\Gamma_R\cup\Gamma_S)$. 

First suppose that $|\xi_c|\leq4K$.
Since $|\xi_a|,|\xi_b|\leq K$, each of the four
frequencies $z_j$ satisfies $|\xi_{z_j}|\leq6K$.
Thus, this contribution is bounded by
$\mathcal I\bigl(F\mathbf{1}_{\{|\xi|\leq6K\}}\bigr)
\lesssim_{K,C_0}\|F\|_2^4$, from Lemma~\ref{lem:bounded-strip}.

Now suppose that $|\xi_c|>4K$.
We first consider the contribution of the first term
in \eqref{eq:case2-pairing}.
Set $s=a+c$. Then
\[
|\xi_s|\geq|\xi_c|-|\xi_a|>3K.
\]
Since the integrand is nonnegative, we may discard
the conditions $a\notin\Gamma_R\cup\Gamma_S$
and $|\xi_{s-a}|>4K$, while retaining $|\xi_s|>3K$.
Changing variables from $(a,b,c)$ to $(a,b,s)$,
it therefore suffices to estimate
\begin{equation}
\label{eq:case3-integral}
\iint_{\substack{
b\in\Gamma_0\setminus(\Gamma_R\cup\Gamma_S)\\
|\xi_s|>3K}}
F(s+b)^2F(s-b)^2
\left(
\int
\mathbf{1}_{\{|\xi_a|\leq K\}}
\mathbf{1}_{\{|T(a,b,s-a)|\leq C_0\}}\dd a
\right)\dd b\dd s.
\end{equation}
Thus, it is enough to show that
\begin{equation}
\label{eq:case3-fibre}
\left|
\left\{
a:|\xi_a|\leq K,\ |T(a,b,s-a)|\leq C_0
\right\}
\right|_{\R\times\frac12\Z}
\lesssim_{K,C_0}1,
\end{equation}
uniformly in $b,s$ with $|R_b|,|S_b|>K$
and $|\xi_s|>3K$.

To prove \eqref{eq:case3-fibre}, write
\[
E_{b,s}
=
\left\{
a=(\xi_a,k_a)\in\R^2:
|\xi_a|\leq K,\ |T(a,b,s-a)|\leq C_0
\right\}.
\]
We first estimate its Euclidean area.
From \eqref{eq:Q-identities}, completing the square gives
\[
T(a,b,s-a)
=
\frac14 Q_b(s)-Q_b\left(a-\frac{s}{2}\right).
\]
Set $v=a-s/2$.
Then the translated set is described by
\[
\xi_v\in I_s,
\qquad
\left|
R_b(\sqrt3\xi_v+k_v)^2
+S_b(\sqrt3\xi_v-k_v)^2-Q_b(s)/4
\right|\leq C_0,
\]
where
$I_s= [-K-\frac{\xi_s}{2},\,
K-\frac{\xi_s}{2} ]$.
Since $|\xi_s|>3K$, we have $\operatorname{dist}(I_s,0) \gtrsim |I_s| \gtrsim K$.
Moreover, $|R_b|,|S_b|>K$.
Applying Lemma~\ref{lem:strip-shell} with $p=R_b$, $q=S_b$
and $I=I_s$, with the fixed sublevel thickness $C_0$,
therefore yields
\begin{equation}
\label{eq:case3-area}
|E_{b,s}|_{\R^2}\lesssim_{K,C_0}1.
\end{equation}
Next, fix the original discrete coordinate $k_a$.
The condition $|\xi_a|\leq K$ restricts $\xi_a$
to an interval of length $2K$. Consequently,
\begin{equation}
\label{eq:case3-section}
\sup_{k_a\in\frac12\Z}
\left|
\left\{
\xi_a:(\xi_a,k_a)\in E_{b,s}
\right\}
\right|
\leq2K.
\end{equation}
Applying Lemma~\ref{lem:mixed} to \eqref{eq:case3-area}
and \eqref{eq:case3-section} gives
\eqref{eq:case3-fibre}. 

Combining \eqref{eq:case3-integral} with
\eqref{eq:case3-fibre}, 
we obtain the desired bound for the contribution
of the first term in \eqref{eq:case2-pairing}.
The second term is treated in the same way after
the change of variables $c\mapsto-c$, which preserves
all the restrictions in this case.
% This completes the proof of Case~3.

\medskip 

This completes the proof.
\end{proof}

\subsection{At most one singular vector}
\label{sec:remaining}

Proposition \ref{prop:strips} controls the contribution to $\mathcal I(F)$
from configurations with at least two singular vectors.
In the complement, at most one of $a,b,c$ is singular,
so at least two are regular (Definition \ref{DEF:reg}). By symmetry, we may take
$b,c$ to be regular, leaving $a$ arbitrary. 
For $d=a+b$ and $e=a-b$, split the interaction set into
\begin{equation*}
\Aone=\left\{
|Q_c(d)-Q_c(e)|\le4C_0,\quad
|T(c,d,e)|^2\le100|Q_c(d)Q_c(e)|
\right\}
\end{equation*}
and
\begin{equation*}
\Atwo=\left\{
|Q_c(d)-Q_c(e)|\le4C_0,\quad
|T(c,d,e)|^2>100|Q_c(d)Q_c(e)|
\right\}.
\end{equation*}
Such a decomposition is inspired by \cite{DFZ} for the hyperbolic Schr\"odinger equation.
The contributions from $\mathcal A_1$ and
$\mathcal A_2$ are controlled by Propositions \ref{prop:direct} and \ref{prop:switched},
which require regularity of $c$ and of $b,c$, respectively.

\subsubsection{The first pairing}
To estimate the contribution of $\mathcal A_1$,
we use the pairing in \eqref{eq:direct-pair}.

\begin{proposition}
\label{prop:direct}
We have
\begin{equation*}
\iiint_{\substack{c\text{ regular}\\(d,e,c)\in\Aone}}
\one_{\{|T(a,b,c)|\leq C_0\}}
\prod_{j=1}^4 F(z_j)\,\dd a\dd b\dd c
\lesssim_{K,C_0}\|F\|_2^4,
\end{equation*}
where $d=a+b$ and $e=a-b$.
\end{proposition}

\begin{proof}
Since $d=a+b$ and $e=a-b$,
by \eqref{eq:interaction} and Young's inequality, we have
\begin{equation}
\label{eq:direct-pair}
\prod_{j=1}^4 F(z_j)
\leq
\frac12 F(d+c)^2F(d-c)^2
+\frac12 F(e+c)^2F(e-c)^2.
\end{equation}
We first consider the contribution of the first term
in \eqref{eq:direct-pair}.
Changing variables from $(a,b,c)$ to $(d,e,c)$, it therefore suffices to estimate
\begin{equation*}
\iint_{c\text{ regular}}
F(d+c)^2F(d-c)^2
\left(
\int \one_{\Aone}(d,e,c)\,\dd e
\right)\dd d\dd c.
\end{equation*}
Thus, to control the contribution of the first term in \eqref{eq:direct-pair}, it is enough to show that, 
\begin{equation}
\label{eq:direct-measure}
\esssup_d
|E_{d,c}|_{\R\times\frac12\Z}
\lesssim_{K,C_0}1,
\end{equation}
uniformly in regular $c$, where
\[
E_{d,c}
=
\left\{
e=(\xi_e,k_e)\in\R^2:
|Q_c(d)-Q_c(e)|\leq4C_0,\quad
|T(c,d,e)|^2\leq100|Q_c(d)Q_c(e)|
\right\}.
\]
We prove \eqref{eq:direct-measure} by considering
separately the two possible signs of $R_cS_c$.

\medskip
\noindent\textbf{Case 1: $R_cS_c>0$.}
We first estimate the Euclidean area of $E_{d,c}$.
Set $X=|R_c|^{1/2}R_e$ and
$Y=|S_c|^{1/2}S_e$.
By \eqref{eq:det},
\[
\left|
\det\frac{\partial(X,Y)}{\partial(\xi_e,k_e)}
\right|
=2\sqrt3\,|R_cS_c|^{1/2}.
\]
Thus, setting $\gamma=\operatorname{sgn}(R_c)Q_c(d)$, we obtain
\[
\begin{aligned}
|E_{d,c}|_{\R^2}
&\leq
\left|
\left\{
e\in\R^2:
|Q_c(d)-Q_c(e)|\leq4C_0
\right\}
\right|_{\R^2}\\
&=
\frac{1}{2\sqrt3\,|R_cS_c|^{1/2}}
\left|
\left\{
(X,Y)\in\R^2:
|X^2+Y^2-\gamma|\leq4C_0
\right\}
\right|_{\R^2}\\
&\lesssim_{C_0}|R_cS_c|^{-1/2}.
\end{aligned}
\]
Here the last inequality follows from the fact that
\begin{align}
\label{eq:shell}
\left|
\left\{
(X,Y)\in\R^2:
|X^2+Y^2-\gamma|\leq4C_0
\right\}
\right|_{\R^2} \les 8\pi C_0,
\end{align}
uniformly in $\gamma$, as can be seen by using polar
coordinates.

Next, fix the original discrete coordinate
$k_e\in\frac12\Z$.
By \eqref{eq:RZ}, we have
\[
Q_c(e)
=3(R_c+S_c)\xi_e^2
+2\sqrt3(R_c-S_c)k_e\xi_e
+(R_c+S_c)k_e^2.
\]
Since $c$ is regular, the leading coefficient
$3(R_c+S_c)$ is nonzero.
Applying \eqref{eq:quadratic} to the first restriction
in the definition of $E_{d,c}$, we obtain
\[
\sup_{k_e\in\frac12\Z}
\left|
\left\{
\xi_e:(\xi_e,k_e)\in E_{d,c}
\right\}
\right|
\lesssim_{C_0}|R_c+S_c|^{-1/2}.
\]

The set $E_{d,c}$ is bounded, since it is contained
in the elliptic shell considered above, and is
semi-algebraic of uniformly bounded complexity.
Applying Lemma~\ref{lem:mixed}, after the fixed
dilation $k_e\mapsto2k_e$, and using the preceding
two estimates, we conclude that
\[
|E_{d,c}|_{\R\times\frac12\Z}
\lesssim_{C_0}
|R_cS_c|^{-1/2}+|R_c+S_c|^{-1/2}
\lesssim_{K,C_0}1.
\]
The last inequality follows from the regularity of $c$.
This proves \eqref{eq:direct-measure} in Case~1.

\medskip
\noindent\textbf{Case 2: $R_cS_c<0$.}
We first assume that $Q_c(d)\neq0$; the exceptional
set will be treated at the end of this case.
As in Case~1, recall
$X=|R_c|^{1/2}R_e$ and
$Y=|S_c|^{1/2}S_e$,
and write
\[
Z=(X,Y),
\qquad
D=(D_1,D_2)
=\bigl(|R_c|^{1/2}R_d,\,
       |S_c|^{1/2}S_d\bigr).
\]
For $U,V\in\R^2$, recall the notation from \cite{DFZ} that
\[
H(U)=U_1^2-U_2^2,
\qquad
H(U,V)=U_1V_1-U_2V_2.
\]
Since $R_c$ and $S_c$ have opposite signs, the two
restrictions defining $E_{d,c}$ become
\begin{equation}
\label{eq:direct-hyperbolic}
|H(D)-H(Z)|\leq4C_0,
\qquad
|H(D,Z)|^2\leq100|H(D)H(Z)|.
\end{equation}
Denote the region defined by
\eqref{eq:direct-hyperbolic} by $\mathcal E_D$.
Notice that $H(D)\neq0$, since $Q_c(d)\neq0$.

The simultaneous dilation
\[
(D,Z)\mapsto
\left(
\frac{D}{2\sqrt{C_0}},
\frac{Z}{2\sqrt{C_0}}
\right)
\]
normalises the first restriction in
\eqref{eq:direct-hyperbolic} to thickness one
and preserves the form of the second restriction.
We may therefore apply the Euclidean estimate
\cite[Proposition 2.4, estimate (2.7)]{DFZ},
whose proof applies to every fixed vector in $\R^2$
with nonzero value of $H$.
Scaling back gives
\[
|\mathcal E_D|_{\R^2}\lesssim C_0,
\]
uniformly in $D$ with $H(D)\neq0$.
Since
\[
\left|
\det\frac{\partial(X,Y)}{\partial(\xi_e,k_e)}
\right|
=2\sqrt3\,|R_cS_c|^{1/2},
\]
we obtain
\begin{align}
\label{eq:4.11}
|E_{d,c}|_{\R^2}
=
\frac{|\mathcal E_D|_{\R^2}}
     {2\sqrt3\,|R_cS_c|^{1/2}}
\lesssim_{C_0}|R_cS_c|^{-1/2}.
\end{align}

Next, fix $k_e\in\frac12\Z$.
As a polynomial in $\xi_e$, the expression
$Q_c(e)-Q_c(d)$ again has leading coefficient
$3(R_c+S_c)$, which is nonzero by the regularity of $c$.
Applying \eqref{eq:quadratic} to the first restriction
in the definition of $E_{d,c}$, we obtain
\begin{align}
\label{eq:4.12}
\sup_{k_e\in\frac12\Z}
\left|
\left\{
\xi_e:(\xi_e,k_e)\in E_{d,c}
\right\}
\right|
\lesssim_{C_0}|R_c+S_c|^{-1/2}.
\end{align}

Applying Lemma~\ref{lem:mixed} to the area estimate
\eqref{eq:4.11} and the section estimate
\eqref{eq:4.12}, we obtain
\[
|E_{d,c}|_{\R\times\frac12\Z}
\lesssim_{K,C_0}1,
\]
uniformly in $d$ with $Q_c(d)\neq0$ and in regular $c$.

It remains to account for the excluded vectors $d$
with $Q_c(d)=0$.
For each fixed $k_d\in\frac12\Z$, this equation is
quadratic in $\xi_d$, with nonzero leading coefficient
$3(R_c+S_c)$, and therefore has at most two solutions.
Consequently, for each regular $c$, the set
$\{d:Q_c(d)=0\}$ has mixed measure zero.
This proves \eqref{eq:direct-measure} in Case~2.

\medskip
We have now established \eqref{eq:direct-measure}
for every regular $c$.
Thus we control the contribution of the first term in \eqref{eq:direct-pair} as required.
The contribution of the second term in
\eqref{eq:direct-pair} is treated in the same way.
Combining the two estimates completes the proof.
\end{proof}

\subsubsection{The second pairing}
To estimate the contribution of $\mathcal A_2$,
we use the pairing in \eqref{eq:switched-pair}.

\begin{proposition}
\label{prop:switched}
We have
\begin{equation}
\label{eq:switched-estimate}
\iiint_{\substack{b,c\text{ regular}\\
(a+b,a-b,c)\in\Atwo}}
\one_{\{|T(a,b,c)|\leq C_0\}}
\prod_{j=1}^4F(z_j)\,\dd a\dd b\dd c
\lesssim_{K,C_0}\|F\|_2^4.
\end{equation}
\end{proposition}

\begin{proof}
By \eqref{eq:interaction} and Young's inequality, we have
\begin{equation}
\label{eq:switched-pair}
\begin{aligned}
\prod_{j=1}^4F(z_j)
\leq{} 
\frac12 F(a+c+b)^2F(a+c-b)^2 +\frac12 F(a-c+b)^2F(a-c-b)^2.
\end{aligned}
\end{equation}
We only consider the contribution from the first term in \eqref{eq:switched-pair}, as the estimate for the second term is similar.
Set $s=a+c$ and define
\[
E_{b,s}
=
\left\{
a=(\xi_a,k_a)\in\R^2:
c=s-a\text{ is regular},\quad
(a+b,a-b,c)\in\Atwo
\right\}.
\]
Changing variables from $(a,b,c)$ to $(a,b,s)$,
it suffices to estimate
\begin{equation*}
\iint_{b\text{ regular}}
F(s+b)^2F(s-b)^2
\left(
\int\one_{E_{b,s}}(a)\,\dd a
\right)\dd s\dd b.
\end{equation*}
Thus, for \eqref{eq:switched-estimate}, it is enough to show that, uniformly in regular
$b$ and arbitrary $s$,
\begin{equation}
\label{eq:switched-measure}
|E_{b,s}|_{\R\times\frac12\Z}
\lesssim_{K,C_0}1.
\end{equation}

Throughout this proof, write
\[
(R_a,S_a)=(x,y),\qquad
(R_b,S_b)=(p,q),\qquad
(R_s,S_s)=(m,n),
\]
and set
\begin{align}\label{eq:lm}
\lambda=\frac{x}{p},
\qquad
\mu=\frac{y}{q}.
\end{align}
These divisions are valid because $b$ is regular.
Since $c=s-a$, we have
\begin{equation}
\label{eq:switched-phase}
T:=T(a,b,s-a)
=px(m-x)+qy(n-y).
\end{equation}
In particular, $|T|\leq C_0$ on $E_{b,s}$.
The change of variables $a\mapsto(x,y)$ has
absolute Jacobian $2\sqrt3$, so Euclidean area
estimates in these coordinates give the corresponding
estimates in the original coordinates.

We first record the section estimate.
For fixed $k_a\in\frac12\Z$, substituting
\[
x=\sqrt3\,\xi_a+k_a,
\qquad
y=\sqrt3\,\xi_a-k_a
\]
into \eqref{eq:switched-phase} makes $T$ a quadratic
polynomial in $\xi_a$ with leading coefficient
$-3(p+q)$.
Since $b$ is regular, this coefficient is nonzero.
Applying \eqref{eq:quadratic}, we obtain
\begin{equation}
\label{eq:switched-section}
\sup_{k_a\in\frac12\Z}
\left|
\left\{
\xi_a:(\xi_a,k_a)\in E_{b,s}
\right\}
\right|
\lesssim_{C_0}|p+q|^{-1/2}
\lesssim_{K,C_0}1.
\end{equation}

We also record a consequence of the quadratic
sublevel estimates for rectangles.
Completing the square in \eqref{eq:switched-phase},
we have
\[
T=\frac{pm^2+qn^2}{4}
-p(x-m/2)^2-q(y-n/2)^2.
\]
For any intervals $I,J$, we have
\begin{equation}
\label{eq:switched-rectangle}
\begin{aligned}
&\left|
\left\{
(x,y)\in I\times J:|T|\leq C_0
\right\}
\right|_{\R^2} \lesssim_{C_0}
|pq|^{-1/2}
\min\bigl\{|p|^{1/2}|I|,\,
          |q|^{1/2}|J|\bigr\}^{1/2}.
\end{aligned}
\end{equation}
To see this, under the change of variables
$X=|p|^{1/2}(x-m/2)$ and
$Y=|q|^{1/2}(y-n/2)$, write
$r=\min\{|p|^{1/2}|I|,|q|^{1/2}|J|\}$.
The transformed region has area $O_{C_0}(r)$
by \eqref{eq:quadratic} and Fubini's theorem,
and $O_{C_0}(1+\log(2+r))$
by \eqref{eq:shell} and \eqref{eq:rectangle}.
Since
$\min\{r,1+\log(2+r)\}\lesssim r^{1/2}$,
restoring the Jacobian $|pq|^{-1/2}$ gives
\eqref{eq:switched-rectangle}.

We now split $E_{b,s}=E_{b,s}^{(1)}\cup E_{b,s}^{(2)}$,
where
\[
E_{b,s}^{(1)}
=
E_{b,s}\cap\{1/2\leq\lambda\mu\leq2\},
\qquad
E_{b,s}^{(2)}
=
E_{b,s}\cap\{\lambda\mu\notin[1/2,2]\}.
\]

\medskip
\noindent\textbf{Case 1: $1/2\leq\lambda\mu\leq2$.}
We first estimate the Euclidean area of
$E_{b,s}^{(1)}$.
Decompose this set into dyadic pieces $E_L$,
$L\in2^\Z$, 
\[
E_{b,s}^{(1)} = \bigsqcup_{L \in 2^\Z} E_L,
\qquad
E_L=E_{b,s}^{(1)}
\cap\{L\leq|\lambda|<2L\}.
\]
On each $E_L$, we have
$L\leq|\lambda|<2L$ and
$\frac1{4L}\leq|\mu|\leq\frac2L$.
Since $\lambda$ and $\mu$ have the same sign,
each piece is contained in the union of two
rectangles in the $(x,y)$-plane with side lengths
$|I| = O(|p|L)$ and $|J| = O(|q|/L)$.
Applying \eqref{eq:switched-rectangle}, we obtain
\[
|E_L|_{\R^2}
\lesssim_{C_0}
|pq|^{-1/2}
\min\left\{
|p|^{3/2}L,\,
\frac{|q|^{3/2}}{L}
\right\}^{1/2}.
\]
Summing the two geometric tails over $L\in2^\Z$
gives
\[
|E_{b,s}^{(1)}|_{\R^2}
\lesssim_{C_0}|pq|^{-1/8}
\lesssim_{K,C_0}1.
\]
Together with \eqref{eq:switched-section} and Lemma~\ref{lem:mixed}, we obtain
\begin{align}\label{eq:E1}
|E_{b,s}^{(1)}|_{\R\times\frac12\Z}
\lesssim_{K,C_0}1.
\end{align}

\medskip
\noindent\textbf{Case 2: $\lambda\mu\notin[1/2,2]$.}
Set
\begin{align}\label{eq:AC}
A=Q_c(a),
\qquad
C=Q_c(b)=p^2(m-x)+q^2(n-y).
\end{align}
Direct calculation gives
\begin{equation}
\label{eq:lambda-identities}
\begin{aligned}
A&=(\lambda+\mu)T-\lambda\mu C,\\
T-\mu C&=(m-x)p^2(\lambda-\mu),\\
T-\lambda C&=(n-y)q^2(\mu-\lambda).
\end{aligned}
\end{equation}
The defining inequality of $\Atwo$ can be written as
\[
|A-C|^2>100\left|(A+C)^2-4T^2\right|.
\]
Consequently,
\[
|A+C|\leq\frac1{10}|A-C|+2|T|.
\]
Using the first identity in
\eqref{eq:lambda-identities}, we obtain
\[
\left(
|1-\lambda\mu|-\frac1{10}|1+\lambda\mu|
\right)|C|
\leq
\left(
2+\frac{11}{10}|\lambda+\mu|
\right)|T|.
\]
Since
\[
|1-h|-\frac1{10}|1+h|
\gtrsim1+|h|
\qquad\text{for }h\notin[1/2,2],
\]
it follows that
\begin{equation}
\label{eq:C-bound}
|C|\lesssim_{C_0}
\frac{1+|\lambda|+|\mu|}
     {1+|\lambda\mu|}.
\end{equation}

Fix a sufficiently large constant $M=M(K,C_0)$.
We consider two cases.

\medskip

\noindent
{\it Case 2.1: $|\lambda|+|\mu|\leq M$.}
Since $x=p\lambda$ and $y=q\mu$, the condition
$|\lambda|+|\mu|\leq M$ implies
\[
(x,y)\in I_M\times J_M,
\qquad
I_M=[-M|p|,M|p|],
\qquad
J_M=[-M|q|,M|q|].
\]
Applying \eqref{eq:switched-rectangle} to these intervals gives
\[
\begin{aligned}
| \{(x,y)\in & I_M\times J_M:
|px(m-x)+qy(n-y)|\leq C_0 \} |_{\R^2} \\
&\lesssim_{C_0}
|pq|^{-1/2}
\min\left\{
|p|^{1/2}|I_M|,
|q|^{1/2}|J_M|
\right\}^{1/2}\\
&\lesssim_{M,C_0}
|pq|^{-1/2}
\min\left\{
|p|^{3/2},|q|^{3/2}
\right\}^{1/2}\\
&\lesssim_{M,C_0}
|pq|^{-1/2}|pq|^{3/8}
=|pq|^{-1/8},
\end{aligned}
\]
where the last inequality uses
$\min\{|p|^{3/2},|q|^{3/2}\}\leq |pq|^{3/4}$.
The image of
$E^{(2)}_{b,s}\cap\{a:|\lambda|+|\mu|\leq M\}$
under $a\mapsto(R_a,S_a)=(x,y)$ is contained in the set
estimated above. Since this change of variables has
absolute Jacobian $2\sqrt3$ by \eqref{eq:det}, we conclude that
\begin{align}\label{eq:Esmall}
\left|
E^{(2)}_{b,s}\cap
\left\{a:|\lambda|+|\mu|\leq M\right\}
\right|_{\R^2}
\lesssim_{K,C_0}|pq|^{-1/8}
\lesssim_{K,C_0}1.
\end{align}
Here the measure is the Euclidean area in the original
$(\xi_a,k_a)$-coordinates, and the final bound follows
from $|p|,|q|>K$.

\medskip

\noindent
{\it Case 2.2: $|\lambda|+|\mu|\geq M$.}
We claim that
\[
\frac12|\lambda|<|\mu|<2|\lambda|.
\]
Indeed, suppose first that $|\lambda|\geq2|\mu|$.
Since $M$ is sufficiently large, $|\lambda|\geq1$,
and \eqref{eq:C-bound} gives
$|\mu C|\lesssim_{C_0}1$.
By the regularity of $c$, we have $|m-x|>K$.
The second identity in \eqref{eq:lambda-identities}
therefore yields
\[
|\lambda-\mu|
\leq
\frac{|T|+|\mu C|}{|m-x|p^2}
\lesssim_{K,C_0}1.
\]
This contradicts
$|\lambda-\mu|\geq|\lambda|/2$
when $M$ is sufficiently large.
The alternative $|\mu|\geq2|\lambda|$ is treated
in the same way using the third identity in
\eqref{eq:lambda-identities}.
This proves the claim.

Since $|\lambda| \sim |\mu|$,
by \eqref{eq:C-bound} we have
$|\lambda C|+|\mu C|\lesssim_{C_0}1$.
Applying the last two identities in
\eqref{eq:lambda-identities}, and using
$|m-x|,|n-y|>K$, we obtain
\begin{equation}
\label{eq:lambda-close}
|\lambda-\mu|
\lesssim_{K,C_0}\frac1{\max(p^2,q^2)} \lesssim_{K,C_0} 1,
\end{equation}
since $b$ is regular and $p = R_b$, $q = S_b$.

Let
\[
E_{\mathrm{large}}
=
E_{b,s}^{(2)}
\cap\{a:|\lambda|+|\mu|>M\}.
\]
Make the affine change of variables
\begin{align}\label{eq:uv}
u=p^2m+q^2n-p^2x-q^2y,
\qquad
v=\frac{x}{p}-\frac{y}{q}.
\end{align}
Since $u=C$ from \eqref{eq:AC} and $v=\lambda-\mu$ from \eqref{eq:lm} and \eqref{eq:uv}, the preceding bounds \eqref{eq:C-bound} and \eqref{eq:lambda-close} imply
\[
|u|\leq C_1,
\qquad
|v|\leq\frac{C_1}{\max\{p^2,q^2\}},
\]
on $E_{\mathrm{large}}$, where $C_1=C_1(K,C_0)$.
Moreover, by the regularity of $b$, we have from \eqref{eq:uv} that
\[
\left| \det\frac{\partial(u,v)}{\partial(x,y)} \right|
=
\left| \frac{p^3+q^3}{pq} \right| \ges K \frac{p^2+q^2}{|pq|}.
\] 
Thus, we obtain
\[
\begin{aligned}
|E_{\mathrm{large}}|_{\R^2}
&\leq
\frac{|pq|}{2\sqrt3\,|p^3+q^3|}
\int_{|u|\leq C_1}
\int_{|v|\leq C_1/\max\{p^2,q^2\}}
\dd v\,\dd u\\ 
&\lesssim_{K,C_0}
\frac{|pq|}
{(p^2+q^2)\max\{p^2,q^2\}}
\lesssim_{K,C_0}1.
\end{aligned}
\]
Together with \eqref{eq:Esmall},
this proves
\begin{align}\label{eq:case2area}
|E_{b,s}^{(2)}|_{\R^2}\lesssim_{K,C_0}1.
\end{align}

As in Case~1, applying Lemma~\ref{lem:mixed} to
$E_{b,s}^{(2)}$ with \eqref{eq:case2area} and
\eqref{eq:switched-section}, we obtain
\begin{align}\label{eq:E2}
|E_{b,s}^{(2)}|_{\R\times\frac12\Z}
\lesssim_{K,C_0}1.
\end{align}
Combining the two cases \eqref{eq:E1} and \eqref{eq:E2} proves
\eqref{eq:switched-measure}.

The contribution of the second term in
\eqref{eq:switched-pair} is treated in the same way.
\end{proof}

\subsection{Completion of the proof}
\label{sec:completion}

\begin{proof}[Proof of Proposition \ref{prop:quartic}]
By Proposition \ref{prop:strips}, the contribution with at least two singular vectors is bounded by $C\norm F_2^4$. In the complement, at least two vectors are regular. 
Propositions \ref{prop:direct} and \ref{prop:switched} bound their contributions. 
\end{proof}

Now, we are ready to prove the main theorem.

\begin{proof}[Proof of Theorem \ref{thm:main}]
We first assume that $f$ is smooth with bounded Fourier support.
By the Fourier reduction \eqref{eq:fourier-reduction}
and Proposition \ref{prop:quartic}, we obtain
\[
\norm{U(t)f}_{L^4([0,1]\times\R\times\T)}^4
\lesssim
\norm f_{L^2(\R\times\T)}^4.
\]
By density, this estimate extends to every
$f\in L^2(\R\times\T)$, proving \eqref{eq:main}.

For $b>1/2$, the embedding \eqref{eq:embedding}
follows from \eqref{eq:main} by the standard transference
principle \cite[Lemma 2.9]{Tao06} and a smooth
partition of unity in time. 
\end{proof}

\subsection*{Acknowledgments}
W.H. was supported by the National Natural Science
Foundation of China (Grant Nos.~12401118 and 12571106)
and the China Scholarship Council
(Grant No.~202508410148).
Y.W. was supported by the EPSRC Mathematical Sciences
Small Grant (Grant No.~UKRI1116).
O.Y. was supported by the Council of Higher Education
(YÖK) of Türkiye through the Young Brains Program
(Genç Beyinler Programı -- GEP) scholarship.

Part of this work was carried out during O.Y.'s visit
to the University of Birmingham from December 2025
to February 2026. O.Y. thanks the University
for its hospitality.

\subsection*{Declaration on the use of generative AI}
The main ideas and the key arguments of this work were developed
between December 2025 and February 2026, during O.Y.'s
visit to the University of Birmingham.
During the subsequent preparation of the manuscript,
the authors used ChatGPT (Codex, GPT-5.6, and GPT-6.0)
to assist with reorganising the proofs, improving
the exposition, and identifying potential gaps
or inconsistencies in the arguments.
The authors have reviewed all AI-assisted revisions,
verified the mathematical arguments, and take full
responsibility for the accuracy and integrity
of the manuscript.

\end{document}